\documentclass[12pt,leqno,fleqn]{amsart}  
\usepackage{amsmath,amstext,amsthm,amssymb,amsxtra}
\usepackage[top=1.5in, bottom=1.5in, left=1.25in, right=1.25in]	{geometry}
\usepackage{txfonts} 
\usepackage[T1]{fontenc}
\usepackage{lmodern}

 \usepackage{euler}   

\usepackage{tikz}
\usepgflibrary{arrows}

\usepackage{mathtools}
\mathtoolsset{showonlyrefs,showmanualtags}

\usepackage{hyperref} 
\hypersetup{
    colorlinks=true,       
    linkcolor=blue,          
    citecolor=magenta,        
    filecolor=magenta,      
    urlcolor=cyan           
}

\usepackage{amsrefs}

\theoremstyle{plain} 
\newtheorem{lemma}[equation]{Lemma} 
\newtheorem{proposition}[equation]{Proposition} 
\newtheorem{theorem}[equation]{Theorem} 
\newtheorem{corollary}[equation]{Corollary} 
\newtheorem{conjecture}[equation]{Conjecture}
\newtheorem{priorResults}{Theorem}

\theoremstyle{definition}

\theoremstyle{remark}

\newtheorem*{ack}{Acknowledgment}

\numberwithin{equation}{section}

\title[] {Sparse Bounds for Spherical Maximal Functions}
\author[M. T. Lacey]{Michael T. Lacey}   

\address{ School of Mathematics, Georgia Institute of Technology, Atlanta GA 30332, USA}
\email {lacey@math.gatech.edu}
\thanks{Research supported in part by grant   National Science Foundation grant DMS-1600693, and by Australian Research Council grant DP160100153.  This material is based upon work supported by the National Science Foundation under Grant No. DMS-1440140, while the authors were in residence at the Mathematical Sciences Research Institute in Berkeley, California, during the Spring 2017 Semester. }

\begin{document}

\begin{abstract}
We consider the averages of a function $ f$ on $ \mathbb R ^{n}$ over spheres of radius $ 0< r< \infty $ given by  $ A _{r} f (x) = \int _{\mathbb S ^{n-1}} f (x- r y) \; d \sigma (y)$, where $ \sigma $ is the normalized rotation invariant  measure on $ \mathbb S ^{n-1}$.  We  prove a sharp range of sparse bounds for  two maximal functions, the first the lacunary spherical maximal function, and the second the full maximal function.  
$$
M _{\textup{lac}} f = \sup _{j\in \mathbb Z } A _{2^j} f ,
\qquad 
M _{\textup{full}} f   = \sup _{ r>0 } A _{r} f .  
$$
The sparse bounds are  very precise variants of the known $L^p$ bounds for these maximal functions. They are derived from known $ L ^{p}$-improving estimates for the localized versions of these maximal functions, 
and the indices in our sparse bound are sharp.  
We derive novel weighted inequalities for weights in the intersection of certain Muckenhoupt and reverse H\"older classes.   
\end{abstract}

\maketitle

\section{Introduction} 
For a smooth function $ f$ on $ \mathbb R ^{n}$, 
let $ A _{r} f (x) = \int _{\mathbb S ^{n-1}} f (x- r y) \; d \sigma (y)$ be the average of $ f$ over the sphere centered at $ x$ and of radius $ r$.  
Here, $ \sigma $ is normalized measure on $ \mathbb S ^{n-1}$.  We consider the two maximal functions 
\begin{align}\label{e:lac}
M _{\textup{lac}} f  & = \sup _{j\in \mathbb Z } A _{2^j} f ,
\qquad 
M _{\textup{full}} f   = \sup _{ r>0 } A _{r} f .  
\end{align}  
The first is the lacunary maximal function, and the second is the full maximal function, introduced by E.~M.~Stein \cite{MR0420116}.   
For both of these, we prove  \emph{sparse bounds.} The latter are particular quantifications of the known $L^p$ inequalities for these operators.  In particular, these bounds quickly imply novel weighted inequalities, for weights in intersections of certain Muckenhoupt and reverse H\"older classes.  
These inequalities are the sharpest known for these operators.

\medskip 

We set notation for the sparse bounds.  
Call a collection of cubes $ \mathcal S$  in $ \mathbb R ^{n}$ \emph{sparse} if there 
are sets $ \{ E_S  \,:\, S\in \mathcal S\}$  
which are pairwise disjoint,   $E_S\subset  S$ and satisfy $ \lvert  E_S\rvert > \tfrac 14 \lvert  S\rvert  $ for all $ S\in \mathcal S$.
For any cube $ Q$ and $ 1\leq r < \infty $, set $ \langle f \rangle_ {Q,r} ^{r} = \lvert  Q\rvert ^{-1} \int _{Q} \lvert  f\rvert ^{r}\; dx  $.  Then the $ (r,s)_m$-sparse form $ \Lambda _{\mathcal S, r,s,m} = \Lambda _{r,s} $, indexed by the sparse collection $ \mathcal S$ is 
\begin{equation} \label{e:sparse_def}
\Lambda _{S, r, s,m} (f,g) = \sum_{S\in \mathcal S} \lvert  S\rvert \langle f  \rangle _{S,r} \langle g  \mathbf 1_{F_S}\rangle _{S,s}.  
\end{equation}
Here, the subscript $ {}_m$ is a reminder that the form has a maximal function component: The sets $ \{F_S \;:\; S\in \mathcal S\}$ are a collection of pairwise disjoint sets with $ F_S\subset S$ for all $ S \in \mathcal S$ (with no requirement on a lower bound on the measure of $ F_S$).   
If there is no subscript $ {}_m$, we mean the same bilinear form, but with $ \mathbf 1_{F_S} \equiv \mathbf 1_{S}$ for all cubes $ S$. The  sparse collection $\mathcal{S}$ is also frequently suppressed in the notation.

Given a  sublinear operator $ T$, and $ 1\leq r, s < \infty$, we set 
$ \lVert T \,:\, (r,s)_m\rVert$ to be the infimum over constants $ C$ so that for all  all bounded compactly supported functions $ f, g$, 
\begin{equation}\label{e:SF}
\lvert  \langle T f, g \rangle \rvert \leq C \sup  \Lambda _{r,s,m} (f,g), 
\end{equation}
where the supremum is over all sparse forms.  
It is essential that the sparse form be allowed to depend upon $ f $ and $ g$. But the point is that the sparse form itself varies over a class of operators with very nice properties.   

We include a discussion of the  lacunary maximal operator for pedagogical reasons.  The following $L^p$ bounds are well known. 

\begin{priorResults}\cites{MR1567040,MR537803}  For all $1<p<\infty$, and dimensions $n$,
we have $\lVert M _{\textup{lac}} \,:\, L^p \mapsto L^p\rVert < \infty$. 
\end{priorResults}

The proofs for the result above compare to the Hardy-Littlewood maximal function, and pass through a square function. For the sparse bound, we will argue directly. 
The bounds below contains the $L^p$ bounds as a trivial corollary, and so it represents a new proof of this fact, one that is intrinsic, in that it only uses properties of spherical averages.

\begin{theorem}\label{t:lac}
Let $ \mathbf L_n$ be the triangle with vertexes $ (0,1)$, $ (1,0)$ and $ (\frac n {n+1}, \frac n {n+1})$. 
(See Figure~\ref{f:L}.) 
For $ n\geq 2$, and 
all  $(\frac 1r, \frac 1s)$  in the interior of  $\mathbf L_n$, we have the inequality  
\begin{equation}  \label{e:LAC}
\lVert M _{\textup{lac}}  : (r, s)_m \rVert < \infty . 
\end{equation} 
Moreover, for $ \frac 1r+\frac1s >1$ not in the closed set  $  \mathbf L_n$, the inequality \eqref{e:LAC} fails. 
\end{theorem}

\medskip 

The case of the full maximal operator is more delicate. The foundation al work is due to E.~M.~Stein, in dimensions $n\ge3$,  and Bourgain in the delicate case of $n=2$. 

\begin{priorResults}\cites{MR1567040,MR537803}  For   and dimensions $n \geq 2$,  we have  
\begin{equation}
\lVert M _{\textup{full}} \,:\, L^p \mapsto L^p\rVert < \infty, \qquad \tfrac {n} {n-1} < p < \infty.
\end{equation}
\end{priorResults}

 The sparse bound below is again a very precise refinement of the well known inequalities above. 

\begin{theorem}\label{t:full} For $ n\geq 2$ and   let  
 $ \mathbf F _n$ be the trapezium   with vertexes $ P_1 =(0,1)$, $P_2= (\frac {n-1}n, \frac 1n)$,  $P_3 = (\frac {n-1} {n},  \frac {n-1}n)$,
 and $ P_4 = (\frac {n ^2 -n} {n ^2 +1}, \frac {n ^2 -n+2} {n ^2 +1})$.   (See Figure~\ref{f:F}.) 
 For all $ (\frac1r,\frac1s)$ in the interior of $ \mathbf F_n$, there holds 
 \begin{equation}\label{e:full}
\lVert M _{\textup{full}}  : (r,s)_m \rVert < \infty . 
\end{equation}
Moreover, for $ \frac 1r+\frac1s >1$ not in the closed set  $  \mathbf F _n$, the inequality \eqref{e:full} fails. 
\end{theorem}

One of the great advantages of sparse bounds is that one can easily derive weighted inequalities for sparse operators, indeed inequalities with sharp dependence upon the Muckenhoupt and reverse H\"older constants. We will discuss this in detail in \S \ref{s:weighted}. 
Weighted inequalities for the spherical maximal function  in the category of  Muckenhoupt and  reverse H\"older classes 
has been studied in \cites{MR1922609,MR1373065}.  We recover and extend their results using the sparse bound. See for instance Proposition~\ref{p:n}

Sparse bounds for different operators is a recent topic of research. These arguments have delivered the most powerful known proof \cite{MR3625108} of the $ A_2$ conjecture. They quickly prove sharp weighted estimates for commutators 
 \cite{160401334}. In other settings, they establish weighted inequalities \cite{160305317} for the bilinear Hilbert transform, as well as other objects in phase plane analysis \cite{161203028}.  Some of these arguments are rather short and elegant, using familiar $ T T ^{\ast} $ style arguments  \cite{2016arXiv161209201C} to provide remarkably sharp control of rough singular integrals.  
 Also see \cites{2016arXiv160901564K,2017arXiv170105170L,2017arXiv170105249K} for further work in this direction.  
In the setting of Radon transforms, the paper \cite{2016arXiv161208881C} discusses a particular arithmetic example, showing that sparse bounds are possible in that setting. Random examples have been considered  in \cites{160908701,160906364,2016arXiv161004968K}.   This paper proves the \emph{first sparse bounds for a Radon transform in the continuous case.}

Our sparse bounds are sharp in the scale of $ L ^{p}$ averages. Sharper results can be obtained using local Lorentz-Orlicz averages at the endpoint cases. The latter is the focus of the article of Richard Oberlin \cite{170404297}. 
Given the close association between sparse bounds and weighted inequalities in other settings, one then suspects that the weighted inequalities that follow are the best possible in the category of Muckenhoupt and reverse H\"older classes.  
In another direction, the core innovation is the identification of the central role of the $L^p$ improving inequalities.  The sharp range of improving inequalities are known for a wide range of Radon transforms. 
Many of these can now be extended to sparse bounds for allied maximal functions.

\bigskip 
We prove the sparse bounds for $ M _{\textup{lac}}$ first, followed by that for $ M _{\textup{full}}$.
Both use the same tool, the $ L ^{p} $ improving mapping properties of the unit scale version of the maximal operators. 
In fact, we need a `continuity' version of these inequalities.   These appear to be new, and  are proved in \S \ref{s:continuity}.    
Once the continuity inequalities are established, the remaining argument is a  variant, but not a corollary, of the innovative paper of Conde, Culiuc, Di Plinio and Ou \cite{2016arXiv161209201C}. 
The argument is presented in detail. 
We then turn to the consequences for weighted inequalities in \S~\ref{s:weighted}.  A final section includes various complements.

\bigskip 
\begin{ack}
 It is a pleasure to acknowledge the interest and input of several people: 
Laura Cladek, Francesco Di Plinio, Richard Oberlin,  Yumeng Ou, and Betsy Stovall, as well as the anonymous referee.   Luz Roncal pointed out an oversight concerning interpolation.  
 \end{ack}


\section{The Lacunary Case} 

The argument has two components, one being a (small) improvement to the classical  $ L ^{p}$-improving properties of the spherical averages due to Littman \cite{MR0358443} and Strichartz \cite{MR0256219}.  
We set $ \mathbf L_n $ to be the triangle with vertexes $ (0,1)$, $ (1,0)$ and $ (\frac {n} {n+1}, \frac n {n+1})$. 
Consider the dual  to $ \mathbf L_n$, defined by $ \mathbf L_n '= \{(\frac1p, \frac1q) :  (\frac1p , 1- \frac1q) \in \mathbf L_n\}$.  See Figure~\ref{f:L}.  

\begin{priorResults}\label{t:LS}\cites{MR0358443,MR0256219} For any point $  (\tfrac 1r, \tfrac 1s)$ in the closed triangle $ \mathbf L_n'$, there holds 
\begin{equation}\label{e:LS}
\lVert A_1     : L ^{r} \mapsto L ^{s} \rVert < \infty . 
\end{equation}
\end{priorResults}

The inequality strengthens as $ s$ increases. In particular the critical case is vertex  $  (\tfrac 1r, \tfrac 1s)= (\frac n {n+1},  \frac 1 {n+1})$.  
The improvement is a `continuity' condition, namely the inequality is preserved, with a small gain, under small translations. 
Let $ \tau _y f (x) =f (x-y)$ be the translation of $ f$ by $ y$.   

\begin{theorem}\label{t:littman}  
Let $ \mathbf L'_n$ be the closed triangle  with vertexes $ (0,0), (1,1)$ and $ (\frac n {n+1}, \frac 1 {n+1})$. 
For $ (\frac 1{r}, \frac1{s}) $ in the interior of $ \mathbf L'_n $  we have the inequalities
\begin{equation}\label{e:littman}
\lVert A_1   - \tau _y A_1  : L ^{r} \mapsto L ^{s} \rVert 
\lesssim \lvert  y\rvert ^{\eta }, \qquad   \lvert  y\rvert \leq 1   , 
\end{equation}
for a choice of $ \eta = \eta (n, r,s)>0$.  
\end{theorem}

A proof is presented in \S \ref{s:continuity}.   
We need a scale invariant version of the inequalities above, which is very easy to prove by a change of variables.

\begin{figure}

\begin{tikzpicture}

\begin{scope}[scale=1.5]
\draw[->,thick] (-.5,0) -- (2.5,0) node[below] {$ \frac1r$} ;
\draw[->,thick] (0,-.5) -- (0,2.75) node[left] {$ \frac1s$}; 
\filldraw[gray, opacity=.2] (0,2) -- (2,0)   --  (1.8,1.8)  -- (0,2) ; %
\draw (1.8,1.8) node[above] { $  (\frac n {n+1},\frac n {n+1})$}; 
\draw (2,0) node[below] {1}; \draw (0,2) node[left] {1}; 
\draw (1.4,1.4) node {$ L_n$}; 

\begin{scope}[xshift=4.5cm] 
\draw[->,thick] (-.5,0) -- (2.5,0) node[below] {$ \frac1r$} ;
\draw[->,thick] (0,-.5) -- (0,2.75) node[left] {$ \frac1s$}; 
\filldraw[gray, opacity=.2] (0,0) -- (2,2)   --  (1.8,.2)  -- (0,0);  
\draw (2,2) node[above] {$(1,1)$}; 
\draw (1.8,.2) node[right] {$ (\frac n {n+1},\frac 1 {n+1})$}; 
\draw (1.4,.6) node {$ L_n'$}; 
\end{scope}

\end{scope}

\end{tikzpicture}

\caption{The triangle $ L_n$ on the left, and $ L'_n $ on the right.}
\label{f:L}
\end{figure}
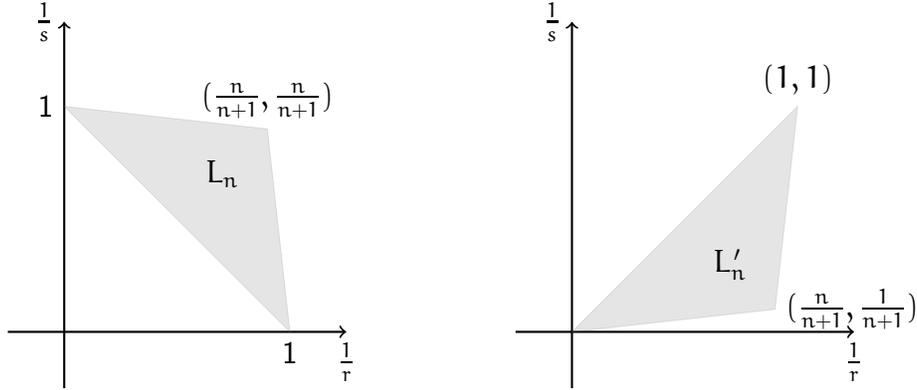

\begin{lemma}\label{l:lac_scale} Let $ f_1 , f_2$ be supported on a cube $ Q$, and let $t \simeq \ell Q$.  
For $ (\frac1r, \frac1s)$ as in  the interior of $\mathbf L_n$,  there holds 
\begin{equation}\label{e:littman2}
\lvert  \langle  A_t f  -  A_t \tau _y f_1 , f_2  \rangle \rvert \lesssim   \lvert  y/ \ell Q \rvert ^{\eta } 
 \lvert  Q\rvert \langle f_1  \rangle_ {Q,r}   \langle f_2  \rangle_ {Q,s} \qquad  \lvert  y\rvert \leq \ell Q.     
\end{equation}
\end{lemma}

We set some notation for the statement of the main lemma. 
For a cube $ Q$ with side length $ 2 ^{q}$, for $ q\in \mathbb Z $, let 
\begin{equation*}
A_Q f = A _{2 ^{q-2}} (f \mathbf 1_{\frac 13 Q}).  
\end{equation*}
It is important to the proof below that the support of $ A_Q f $ is contained in $ Q$.  
There are a choice of $ 3 ^{n}$ dyadic grids $ \mathcal D_1 ,\dotsc, \mathcal D _{3 ^{n}} $ so that 
\begin{equation*}
A _{2 ^{q-2}} f = \sum_{t=1} ^{3 ^{n} } \sum_{Q \in \mathcal D_t : \ell Q = 2 ^{q}} A_Q f.  
\end{equation*}
Therefore, it suffices to prove the sparse bound for each of the maximal operators 
\begin{equation} \label{e:MD}
M _{\mathcal D_t} f := 
\sup _{Q\in \mathcal D_t} A_Q f , \qquad  1\leq t \leq  3 ^{n}.  
\end{equation}
The specific dyadic grid in question is immaterial, so we fix such a grid below, and write $ \mathcal D = \mathcal D_t$.  
This is the kernel of the proof.   Notice that the first function is an  indicator function.

\begin{lemma}\label{l:lac_sparse} 
Let $ 1< r ,s  < \infty $ be as in Theorem \ref{t:lac},  and let $ C_0 >1$ be a constant.   
Let  $ f_1 = \mathbf 1_{F_1}$.  
Let $ \mathcal Q$ be a collection of sub cubes of $ Q_0 \in \mathcal D$ for which there holds 
\begin{equation}
\sup _{Q'\in \mathcal Q} \sup _{Q \;:\; Q'\subset Q\subset Q_0}   \Bigl\{\frac { \langle f_1  \rangle_{Q,r}} 
{\langle f_1  \rangle_{Q_0,r}} +  \frac { \langle f_2  \rangle_{Q,s}} 
{\langle f_2  \rangle_{Q_0,s}} \Bigr\} < C_0. 
\end{equation}
Then, there holds  
\begin{equation}\label{e:lac_sparse}
\Bigl\langle  \sup _{Q\in \mathcal Q}  A_Q f_1 ,f_2 \Bigr\rangle 
\lesssim \lvert  Q_0\rvert \langle f_1  \rangle_{Q_0,r} \langle f_2 \rangle_{Q_0,s}.     
\end{equation}
\end{lemma}

\begin{proof}
The supreumum is linearized. Thus, for pairwise disjoint sets  $ \{ F_ Q \;:\; Q\in \mathcal Q\} $ with $ F_Q \subset Q$, set $f_Q = f_2 \mathbf{1}_{F_Q}$.  We estimate 
\begin{align} \label{e:Lam}
\sum_{Q\in \mathcal Q} \langle   A_Q f_1 , \mathbf 1_{F_Q} f_2 \rangle . 
\end{align}

We take $ \mathcal B$ to be the maximal dyadic subcubes of $ Q_0$ so that  we have 
\begin{equation}\label{e:bad}
\max  \Bigl\{  \frac{\langle f_1  \rangle_{Q,r}} {\langle f_1  \rangle_{Q_0,r}},\
 \frac{\langle f_2  \rangle_{Q,s}} {\langle f_2  \rangle_{Q_0,s}} \Bigr\}  > 2 C_0   .  
\end{equation}
Perform a standard Calder\'on-Zygmund decomposition on $f_1$.  
Set $ f_1 = g_1 + b_1$ where 
\begin{equation} \label{e:bad1}
b_1 = \sum_{P\in \mathcal B} (f_1 - \langle f_1 \rangle_P) \mathbf 1_{P} = \sum_{k = - \infty } ^{q_0 -1} 
\sum_{P\in \mathcal B (k)}  (f_1 - \langle f_1 \rangle_P) \mathbf 1_{P}  =: 
 \sum_{k = - \infty } ^{q_0 -1} B _{1,k} , 
\end{equation}
where above we write $ \ell Q_0 = 2 ^{q_0}$, and set $ \mathcal B (k) = \{P\in \mathcal B : \ell P= 2 ^{k}\}$.  

The bilinear expression in \eqref{e:Lam} is dominated by a sum of  two terms.  The first places the good function $ g_1$ in the first place. It is a bounded function, so that
\begin{equation*}
 \sum_{Q\in \mathcal Q} \lvert  \langle   A_Q g_1 , f_Q \rangle \rvert \lesssim  
\sum_{Q\in \mathcal Q}  \lVert f_2 \mathbf 1_{F_Q}\rVert_1  \lesssim \lvert  Q_0\rvert.    
\end{equation*}
This just depends upon the disjointness of the sets $ F_Q$.  

\smallskip 

The second has $ b_1$ in the first position.  
We have this following easy, but essential, fact:  For all $ Q\in \mathcal Q$ and $ P\in \mathcal B$, if $ Q\cap P \neq \emptyset $, then $ P\subsetneq Q$.  Therefore, for any $ Q\in \mathcal Q$, with $ \ell Q = 2 ^{q}$, 
we have, using the notation of \eqref{e:bad1}, 
\begin{equation*}
 \langle   A_Q b_1 , f_Q \rangle= \sum_{k \;:\; k<q} 
  \langle   A_Q B _{1,k} , f_Q \rangle = \sum_{k=1} ^{\infty }  \langle   A_Q B _{1,q-k} , f_Q \rangle. 
\end{equation*}
Therefore, 
\begin{align*}
\Bigl\lvert 
 \sum_{Q\in \mathcal Q} \langle   A_Q b_1 , f_Q \rangle   
\Bigr\rvert&\leq 
\sum_{k=1} ^{\infty }\sum_{Q\in \mathcal Q}\lvert   \langle   A_Q B_{1, q-k} , f_Q \rangle\rvert , 
& (\ell Q = 2 ^{q})
\end{align*}

We achieve the desired bound, with geometric decay in $ k$, derived from our continuity inequalities.    For $Q\in \mathcal{Q}$, with $\ell Q=2^q$, we estimate as follows, using the mean zero properties of the bad functions. 
\begin{align}
\lvert  \langle   A_Q B_{1, q-k} &, f_Q\rangle 
= \lvert  \langle    B_{1, q-k} , A_Q ^{\ast} f_Q\rangle \rvert  
\\& 
= \sum _{P\in \mathcal B (q-k)}
\frac 1 { \lvert P\rvert} 
\Bigl \lvert \int_P \int_P 
[ A_Q ^{\ast}  f_Q (x) - A_Q ^{\ast} f_Q (x')]
 \cdot    B _{1,q-k} (x)  \; dx \, dx'\Bigr\rvert 
 \\
 & \lesssim 
\frac 1 { \lvert P_0\rvert} 
\Bigl \lvert  \int
[ A_Q ^{\ast}  f_Q (x) -  \tau_y A_Q ^{\ast} f_Q (x)]
 \cdot    B _{1,q-k} (x)  \; dx \Bigr\rvert  \, dy
 \\  \label{e:zcont}
 & \lesssim 2 ^{- \eta k}  \lvert  Q\rvert  
\langle B _{1,q-k} \mathbf 1_{Q} \rangle _{Q, r} 
 \langle f_Q \rangle_{Q, s}  . 
\end{align}
Above, $ P_0$ is the cube of side  length $ 2 ^{q-k+1}$ centered at the origin, and we use our continuity inequality \eqref{e:littman}. 

It remains to argue that uniformly in $ k \geq 1$, 
\begin{equation} \label{e:SQ}
\sum_{q}\sum_{\substack{Q\in \mathcal Q \\ \ell Q= 2 ^{q} }} 
 \lvert  Q\rvert  
\langle B _{1,q-k} \mathbf 1_{Q} \rangle _{Q, r} 
\langle f_2 \mathbf 1_{F_Q} \rangle_{Q, s}  \lesssim  \langle f_1 \rangle _{Q_0, r} \langle f_2\rangle _{Q_0, s} \lvert  Q_0\rvert.  
\end{equation}
This follows from the  (a) disjointness of the sets $ F_Q$,  (b) the disjointness of the supports of $ B _{1,k} \mathbf 1_{Q}$, for $ k\geq 1$ fixed, and (c) $ 1/r+1/s\geq 1$.  
It is at this point that we need the $ f_1$  to be and  indicator function. 

Fix the integer $ k$.  Dominate 
\begin{equation*}
\lvert   B _{1,q-k}  \rvert  \lesssim \langle f_1 \rangle _{Q_0,r} \mathbf 1_{E_q} 
+ \mathbf 1_{F _{1, q}}, 
\end{equation*}
where the $ E_q = E _{q,k}$ are pairwise disjoint sets in $ Q_0$ as $ q$ varies, 
and $ F _{1,q} = F _{1,q,k}$ are pairwise disjoint sets in $ F_1$.  
This leaves us two terms to control. 

The first term requires us  to  show that 
\begin{equation}\label{e:SQ1}
 \langle f_1 \rangle _{Q_0,r} 
\sum_{q}\sum_{\substack{Q\in \mathcal Q \\ \ell Q= 2 ^{q} }} 
 \lvert  Q\rvert  
\langle  \mathbf 1_{E_q} \rangle _{Q, r} 
\langle f_2 \mathbf 1_{F_Q} \rangle_{Q, s}  \lesssim  \langle f_1 \rangle _{Q_0, r} \langle f_2\rangle _{Q_0, s} \lvert  Q_0\rvert.  
\end{equation}
Notice that the term $  \langle f_1 \rangle _{Q_0,r}$ appears on both sides.  
 Now, this is easy to see from H\"older's inequality   for $ 1/r + 1/s=1$.  
 In the case that $ 1/r + 1/s = 1+ \tau >1$, set $ 1/ \dot r = 1/r - \tau  $.  Then,   
 $ 1/ \dot r + 1/s=1$, and  \( r < \dot r\), so that  
\begin{equation*}
\langle  \mathbf 1_{E_q} \rangle _{Q, r} 
\langle f_2 \mathbf 1_{F_Q} \rangle_{Q, s}  
\lesssim  
\langle  \mathbf 1_{E_q} \rangle _{Q, \dot r} 
\langle f_2 \mathbf 1_{F_Q} \rangle_{Q, s} 
\end{equation*} 
But, then \eqref{e:SQ1} follows from the case of duality,

The second term requires us  to  show that 
\begin{equation}\label{e:SQ2}
\sum_{q}\sum_{\substack{Q\in \mathcal Q \\ \ell Q= 2 ^{q} }} 
 \lvert  Q\rvert  
\langle  \mathbf 1_{F_{1,q}}  \rangle _{Q, r} 
\langle f_2 \mathbf 1_{F_Q} \rangle_{Q, s}  \lesssim  \langle f_1 \rangle _{Q_0, r} \langle f_2\rangle _{Q_0, s} \lvert  Q_0\rvert.   
\end{equation}
The inequality holds in the case of $ 1/r+1/s=1$. 
For $ 1/r + 1/s = 1+  \tau $, define $ \dot r$  as above.  
The point is then that
\begin{equation*}
\langle  \mathbf 1_{F_{1,q} } \rangle _{Q, r} 
\langle f_2 \mathbf 1_{F_Q} \rangle_{Q, s} \lesssim 
\langle  \mathbf 1_{F_{1} } \rangle _{Q_0} ^{\tau } 
\cdot 
\langle  \mathbf 1_{F_{1,q} } \rangle _{Q, \dot r} 
\langle f_2 \mathbf 1_{F_Q} \rangle_{Q, s'}   
\end{equation*}
where we have used the  the stopping condition \eqref{e:bad}.  
From this, we conclude \eqref{e:SQ2}, and so \eqref{e:SQ}. The proof is complete. 
 (The stopping condition is non linear, preventing a general interpolation argument at this point. That is why we passed to indicator sets.) 

\end{proof}

\begin{proof}[Proof of Theorem \ref{t:lac}]  
We deduce the m-sparse bound for the operator $ M _{\mathcal D}$ in \eqref{e:MD}.  
From this it follows that $M _{\textup{lac}}$ is bounded by the sum of a finite number of sparse forms. 
But, the principle described in \eqref{e:ONE}  shows that there is a constant $C$, so that given $f,g$, 
there is a fixed sparse form $\Lambda _{\mathcal{S}_0, r,s}$, so that 
\begin{equation}
\sup_{\mathcal{S}} \Lambda _{\mathcal{S}, r,s}(f,g)\leq C\Lambda _{\mathcal{S}_0, r,s}(f,g). 
\end{equation}
Thus, the sparse bound as claimed will follow. 

\smallskip 

The main line of the argument comes in two stages. The first stage is to prove the sparse 
bound for \( f_1 = \mathbf 1_{F_1}\), and the second stage is for $ f_1$ a general function.  
The beginning of both stages is the same.  
We can assume that $ f_1, f_2$ are bounded functions supported on a dyadic cube $Q_0 \in \mathcal D $. 
Indeed, we can even assume that for any cube $Q\supsetneqq Q_0$, we have $A_Q f \equiv 0$.  Namely, for the construction of the sparse bound, we need only consider cubes $Q\subset Q_0$.  

We then add the cube $ Q_0$ to $ \mathcal S$. We take the $\mathcal S$-children of $ Q_0$ to be the collection $ \mathcal E$ of maximal children $P\subsetneq  Q_0$ for which $ \langle f_1 \rangle _{P,r} > C_n  \langle f_1 \rangle _{Q_0,r} $, 
or $  \langle f_2 \rangle _{P, \sigma } > C_n  \langle f_2 \rangle _{Q_0, \sigma }$. Here $ \sigma >s$ should 
satisfy $ 1/r+ 1/ \sigma >1$, but is otherwise arbitrarily close to $ s$.   
  Let $ E$ be the union of these maximal children.  For a choice of constant $C_n>1$, we have $\lvert E\rvert < \tfrac 12 \lvert Q_0\rvert$.   Set 
  $ \mathcal Q = \{  P\subset Q_0 : P\not\subset E\}$.  Associated to the set $ Q_0$ we need the set 
\begin{equation*}
F_{Q_0} = \{ x\in Q_0 \;:\;    M _{\mathcal D} f (x) = \sup _{Q\in \mathcal Q} A_Q f (x)\}.  
\end{equation*}
We need to see that 
\begin{equation}\label{e:22PP}
\Bigl\langle \sup _{Q\in \mathcal Q} A_Q f (x), f_2 \mathbf 1_{F _{Q_0}}  \Bigr\rangle  
\lesssim 
\langle f_1 \rangle _{r, Q_0} 
  \langle f_2 \mathbf 1_{F _{Q_0}} \rangle _{\sigma , Q_0} \lvert  Q_0\rvert
\end{equation}
A straight forward recursion then completes the proof of the sparse bound.

\smallskip   
The second stage of the argument is to allow \( f_1\) to be an arbitrary bounded function. 
We can assume that  $ f_1 =  \sum_{k \in \mathbb Z } 2 ^{-k} \mathbf 1_{F _{1,k}}$, for disjoint sets $ F _{1,k}$.    
 Apply \eqref{e:22PP}  for each set $ F _{1,k}$.   
For each $ k\in \mathbb Z $, we get a sparse collection $ \mathcal S_k$ so that 
\begin{align}
\Bigl\langle \sup _{Q\in \mathcal Q} A_Q f (x), f_2 \mathbf 1_{F _{Q_0}}  \Bigr\rangle  
& \lesssim 
\sum_{k \in \mathbb Z }  2 ^{-k}
\sum_{S\in \mathcal S_k}  
\langle \mathbf 1_{F _{1,k}} \rangle _{r, S} 
  \langle f_2 \mathbf 1_{F _{Q_0}} \rangle _{\sigma , S} \lvert  S\rvert 
  \\
  & \lesssim 
    \langle f_2 \mathbf 1_{F _{Q_0}} \rangle _{\sigma , Q_0} 
\sum_{k \in \mathbb Z }  2 ^{-k}
\sum_{S\in \mathcal S_k}  
\langle \mathbf 1_{F _{1,k}} \rangle _{r, S} 
\lvert  S\rvert 
\\  \label{e:sumk}
& \lesssim  \lvert  Q_0\rvert 
 \langle f_2 \mathbf 1_{F _{Q_0}} \rangle _{\sigma , Q_0} 
\sum_{k \in \mathbb Z }  2 ^{-k}
\langle \mathbf 1_{F _{1,k}} \rangle _{\rho , Q_0}  . 
\end{align}
Above, we have again used the stopping condition to move the $ f_2$ term outside the sum, 
and then applied \eqref{e:CM}, where $ \rho > r$  satisfies $ 1/ \rho + 1/ \sigma >1$, and is arbitrarily 
close to $ r$.  It remains to bound the  sum over $ k$ in \eqref{e:sumk}. But note that  
we can compare to the $ L ^{\rho ,1}$ Lorentz norm: 
\begin{align*}
\sum_{k \in \mathbb Z }  2 ^{-k}
\langle \mathbf 1_{F _{1,k}} \rangle _{\rho , Q_0} 
\simeq 
\lVert   f_1 \;:\; L ^{\rho ,1} (Q_0 ; dx / \lvert  Q_0\rvert )\rVert  \lesssim \langle f_1 \rangle _{\rho ' ,Q}, 
\qquad 1 < \rho ' < \rho .  
\end{align*}
We can pass from $ L ^{\rho ,1}$ to $ L ^{\rho ' }$, since we are working on a probability space. 
And,  so a slightly weaker form of \eqref{e:22PP} holds.  But, that form is sufficient, since 
we claim sparse bounds in an open set. 
\end{proof}

This is an elementary fact.  
\begin{proposition}\label{p:CM} Let $ \mathcal S$ be a collection of sparse subcubes of a fixed 
dyadic cube $ Q_0$, and let $ 1\leq s < t < \infty $. Then, for a bounded function $ \phi $,  
\begin{equation}\label{e:CM}
\sum_{Q\in \mathcal S} \langle \phi  \rangle _{s, Q} \lvert  Q\rvert \lesssim \langle \phi  \rangle _{t, Q_0} \lvert  Q_0\rvert  .       
\end{equation}
 
\end{proposition}

\begin{proof}
This is an instance of the Carleson embedding inequality, combined with the definition of sparsity. 
Below,  $ t' = \frac {t} {t-1}$ is the dual index to $ t$.   
\begin{align*}
\sum_{Q\in \mathcal S} \langle \phi  \rangle _{s, Q} \lvert  Q\rvert 
& =
\sum_{Q\in \mathcal S} \langle \phi  \rangle _{s, Q} \lvert  Q\rvert ^{1/t+ 1/t'} 
\\
& \leq 
\Bigl[ 
\sum_{Q\in \mathcal S} \langle \phi  \rangle _{s, Q} ^{t} \lvert  Q\rvert
\Bigr] ^{1/t} \Bigl[ \sum_{Q\in \mathcal S}  \lvert  Q\rvert \Bigr] ^{1/t'} 
\\
& \lesssim
\Bigl[ 
\sum_{Q\in \mathcal S} \langle  \lvert  \phi \rvert ^{s}  \rangle _{ Q} ^{t/s} \lvert  Q\rvert
\Bigr] ^{1/t}   \lvert  Q_0\rvert ^{1/t'}   
\\
& \lesssim \lVert  \phi  \mathbf 1_{Q_0} \lVert _{t} \lvert  Q_0\rvert ^{1/t'}.  
\end{align*}

\end{proof}

\section{The Full Supremum} 

The analog of the $ L ^{p}$-improving properties of $ A_1$ in Theorem~\ref{t:LS} 
concern the `unit scale'  maximal function 
$
\tilde M f   = \sup _{1\leq t \leq 2} A_t f 
$.
This is due to 
 Schlag \cite{MR1388870}, also see  Schlag and Sogge \cite{MR1432805}.  
 
\begin{priorResults}\label{t:S}
Let $ \mathbf F_n'$ be the closed convex hull of the four points $ P_1' = (0,0)$, $ P_2' = (\frac  {n-1}n,\frac  {n-1}n)$, $ P_3' = (\frac  {n-1}n, \frac 1 {n})$, and $ P_4' = (\frac {n ^2 -n} {n ^2 +1 }, \frac {n-1} {n ^2 +1})$.  For all $ (\frac1r, \frac1s) $ in   $ \mathbf F'_n $, we have 
\begin{align}  \label{e:full_improve}
\lVert 
\tilde M  \;:\; L ^{r}\mapsto L^s\rVert < \infty . 
\end{align}
\end{priorResults}

This  `continuity property' is a corollary.

\begin{theorem}\label{t:full_improve}  
For all $ (\frac1r, \frac1s) $ in the interior of $ \mathbf F'_n $, we have for some  $ \eta = \eta (n,r,s)>0$, 
\begin{align}  \label{e:full_improve_cont}
\lVert  \sup _{1\leq t \leq 2} \lvert  A_t f - \tau _y A_t f  \rvert 
\rVert_s 
\lesssim  \lvert  y\rvert ^{\eta } \lVert f\rVert_r, \qquad  \lvert  y\rvert<1 .
\end{align}
\end{theorem}

We will delay the proof of this theorem to the next section.  See Figure~\ref{f:F} for a picture of the trapeziums $\mathbf  F_n$ and $\mathbf  F'_n$.

\begin{figure}

\begin{tikzpicture}

\begin{scope}[scale=1.5]
\draw[->,thick] (-.5,0) -- (2.5,0) node[below] {$ \frac1p$} ;
\draw[->,thick] (0,-.5) -- (0,2.5) node[left] {$ \frac1q$}; 
\filldraw[gray, opacity=.2] (0,2)  -- (1.9,.1)  -- (1.9, 1.5 ) --  (1.4,1.8) -- (0,2);
\draw (0,2)  node[left] {$ P_1$}; 
\draw (1.9,.2) node[above,right] {$ P_2$};  \draw (1.4,1.8)  node[above] {$ P_4$}; 
\draw  (1.9, 1.5 )  node[right] { $ P_3$}; 
\draw (1.4,1.1) node {$ F_n$}; 

\begin{scope}[xshift=4cm] 
\draw[->,thick] (-.5,0) -- (2.5,0) node[below] {$ \frac1p$} ;
\draw[->,thick] (0,-.5) -- (0,2.5) node[left] {$ \frac1q$}; 
\filldraw[gray, opacity=.2] (0,0) -- (1.9,1.9) -- (1.9, .5 ) --  (1.4,.2) -- (0,0);  
\draw (-.3,-.1) node[below] {$ P_1'$}; 
\draw (1.9,1.9) node[above] {$ P_2'$}; 
\draw (1.9, .5 )  node[right] { $ P_3'$}; 
\draw (1.2, .2 )  node[above] { $ P_4'$}; 
\draw (1.6,1) node {$ F_n'$}; 
\end{scope}

\end{scope}

\end{tikzpicture}

\caption{The trapezium $ F_n$ on the left, and $ F'_n $ on the right. (When $ n=2$, they are in fact triangles.)} 
\label{f:F}
\end{figure}
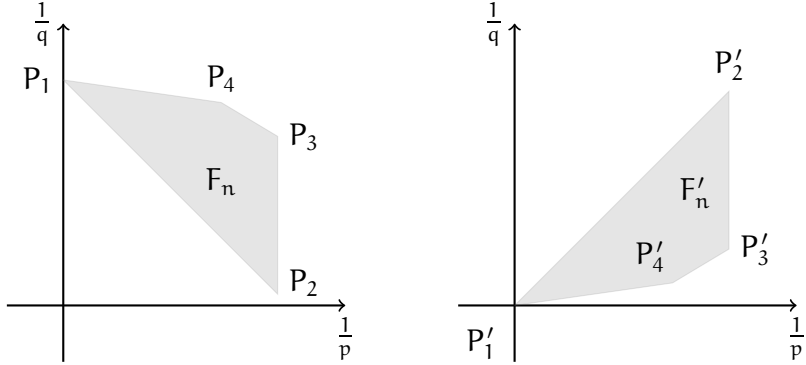

We again make a dyadic reduction.  For a cube $ Q$ with side length $ 2 ^{q}$, for $ q\in \mathbb Z $, let 
\begin{equation*}
\tilde M_Q f =   \sup _{2 ^{q -3} \leq t < 2 ^{q-2}} A _{t} (f \mathbf 1_{\frac 13 Q}), \qquad \ell Q = 2 ^{q}.  
\end{equation*}
There are a choice of $ 3 ^{n}$ dyadic grids $ \mathcal D_1 ,\dotsc, \mathcal D _{3 ^{n}} $ so that 
\begin{equation*}
 \sup _{2 ^{q -3} \leq t < 2 ^{q-2}} A _{t} (f \mathbf 1_{\frac 13 Q}) 
 \leq  \sum_{s=1} ^{3 ^{n} } \sum_{Q \in \mathcal D_s \;:\; \ell Q = 2 ^{q}}  \tilde M_Q f 
\end{equation*}
Therefore, it suffices to prove the sparse bound for each of the maximal operators 
\begin{equation} \label{e:fMD}
M _{\mathcal D_s} f := 
\sup _{Q\in \mathcal D_t} \tilde M_Q f , \qquad  1\leq s \leq  3 ^{n}.  
\end{equation}
We fix such a grid below, and write $ \mathcal D = \mathcal D_s$.  
The main Lemma is as before.   We will prove it, and leave the details of the derivation of Theorem~\ref{t:full} to the reader.  

\begin{lemma}\label{l:full_sparse}
Let $ (\frac1r, \frac1s)$ be in the interior of $ \mathbf F _n$.   Let $ f_1 = \mathbf 1_{F_1}$.  
Let $ \mathcal Q \subset \mathcal{D}$ be a collection of sub cubes of $ Q_0$ so that  
\begin{equation*}
\sup _{Q'\in \mathcal Q}  \sup _{Q \;:\; Q'\subset Q\subset Q_0} \Bigl\{\frac { \langle f_1  \rangle_{Q,r}} 
{\langle f_1  \rangle_{Q_0,r}} +  \frac { \langle f_2  \rangle_{Q,s}} 
{\langle f_2  \rangle_{Q_0,s}} \Bigr\} < C_0. 
\end{equation*}
Then, there holds 
\begin{equation}\label{e:full_sparse}
\Bigl\langle  \sup _{Q\in \mathcal Q}  \tilde M_Q f_1 ,f_2 \Bigr\rangle 
\lesssim \lvert  Q_0\rvert \langle f_1  \rangle_{Q_0,r} \langle f_2 \rangle_{Q_0,s}.     
\end{equation}
\end{lemma}

\begin{proof}
The proof closely follows the lines of the proof of Lemma~\ref{l:lac_sparse}. 
Assume $ \langle f_1  \rangle_{Q_0,r}= \langle f_2 \rangle_{Q_0,s}=1$.  Define the collection of `bad' cubes $ \mathcal B$ as in \eqref{e:bad}.  We bound the bilinear form 
\begin{equation} \label{e:BF}
\sum_{Q\in \mathcal Q} \langle  \tilde M_Q f_1 ,f_Q  \rangle, 
\end{equation}
where  $ \{F_Q \;:\; Q\in \mathcal Q\} $ is a family of disjoint sets with $ F_Q\subset Q$, 
and $ f_Q = \mathbf 1_{F_Q} f_2$.

Use the  Calder\'on-Zygmund decomposition,  just like in   \eqref{e:bad1}. 
The bilinear form in \eqref{e:BF} is a divided into two terms, of which the first has the good function $ g_1$ in the first place.  
\begin{equation*}
\sum_{Q\in \mathcal Q} \lvert  \langle  \tilde M_Q g_1 , \mathbf 1_{F_Q}f_2  \rangle\rvert 
\lesssim \sum_{Q\in \mathcal Q} \int _{F_Q} \lvert  f_2\rvert\;dx \lesssim \lvert  Q_0\rvert.    
\end{equation*}

\smallskip 

The second term has $ b_1$ in the first place, and $ f_Q$ in the second.  Namely, we have to bound 
\begin{equation*}
\sum_{Q\in \mathcal Q}\lvert   \langle  \tilde M_Q b_1 , f_Q  \rangle \rvert 
\leq  \sum_{k=1} ^{\infty } 
\sum_{Q\in \mathcal Q}\lvert   \langle  \tilde M_Q B_{1, q- k}, f_Q  \rangle \rvert , \qquad  ( \ell Q= 2 ^{q}). 
\end{equation*}
Above, we have used the expansion in \eqref{e:bad1}.  
We will use the continuity inequality \eqref{e:full_improve_cont} to establish the desired bound with geometric decay in $ k$. 
Let us argue by duality.  For each $ Q\in \mathcal Q$ we can replace $ \tilde M_Q  $ 
by $ L _{Q} \phi  (x)= A _{t_Q (x)} \phi (x)$, where $ t_Q \;:\; \tfrac 13 Q \mapsto [2 ^{q-2}, 2 ^{q-1}]$ is measurable.  
Then, estimate 
\begin{align}
\lvert \langle  & L _{Q}    B_{1,q- k_1} , f_Q \rangle  \rvert = 
\lvert \langle    B_{1,q- k} ,  L _{Q} ^{\ast} f_Q  \rangle  \rvert 
\\&
\leq \sum_{\substack{P\in \mathcal B (q-k) \\ P\subset Q} } 
\frac 1 {\lvert  P\rvert } \int_{P}
\Bigl  \lvert  \int_P   B_{1,q-k_1} (x)  \cdot  \bigl(   L _{Q} ^{\ast}  f_Q(x)- L _{Q} ^{\ast}   f_Q(x') \bigr)\; dx \Bigr\rvert \,dx' 
\\ \label{e:full3}
& \lesssim 2 ^{- \eta k} 
\lvert  Q\rvert \langle  B_{1,q- k} \rangle_{Q,r} \langle f_Q \rangle _{Q,s} . 
\end{align}
Here, the notation is similar to \eqref{e:zcont}, and we   appeal to the  scale-invariant and \emph{dual} form of   \eqref{e:full_improve_cont}.  
The remainder of the argument is exactly as in the proof of Lemma~\ref{l:lac_sparse}. 

\end{proof}

\section{Proof of  the Continuity Inequalities.  } \label{s:continuity}

\subsection{Proof of Theorem~\ref{t:littman}}
From Plancherel's theorem,   we have 
\begin{equation}\label{e:trans0}
\lVert A_1 f  -  A_1 \tau _y : L ^2 \mapsto L ^2   \rVert
= \lVert (1 - e ^{i y \cdot \xi }) \widehat {d \sigma } (\xi ) \rVert_ \infty 
\lesssim \lvert  y\rvert ^{\eta _0} , \qquad \eta _0 = \eta _o (n) >0. 
\end{equation}
To see this last inequality, we need only appeal to the well known decay estimate for $ \lvert  \widehat {d \sigma } (\xi )\rvert$ which we recall below.

In interpolation between this $ L^2$ estimate and the $ L ^{r}$ improving estimates of Theorem~\ref{t:LS},  it is clear that the conclusion \eqref{e:littman} holds for $ (\frac1r, \frac1s)$ in the interior of the triangle $ \mathbf L'_n$.  

\subsection{Proof of Theorem~\ref{t:full_improve}}
We recall that the Fourier transform of  $ \sigma $, the uniform measure on the sphere $ \mathbb S ^{n-1}$, is 
\begin{equation} \label{e:sig}
\widehat {d \sigma } (\xi ) = e ^{-i \lvert  \xi \rvert } a _{-} (\xi )  + e ^{i \lvert  \xi \rvert } a _{+} (\xi ),  
\end{equation}
where $ \lvert  \partial ^{\alpha } a _{\pm} ( \xi ) \rvert \lesssim (1+ \lvert  \xi \rvert ) ^{- (n-1)/2- \lvert  \alpha \rvert }$.

The trapezium $ \mathbf F'_n$ is contained in the triangle $ \mathbf L'_n$. 
Thus, if $ \mathcal T \subset [1,2]$ is a finite set, it follows from Theorem~\ref{t:littman} that we have 
\begin{equation*}
\lVert \sup _{t\in \mathcal T} \lvert  A _{t} f - \tau _y A _{t} f \rvert \rVert _ {p_2}
\lesssim  {}^{\sharp} (\mathcal T) ^{1/p_2}  \cdot \lvert  y\rvert ^{\eta }   \lVert f\rVert _{p_1} , \qquad (\tfrac1{p_1}, \tfrac 1{p_2}) \in \mathbf F'_n \setminus \{  (0,1), (1,0)\}. 
\end{equation*}
Taking $ \mathcal T $ be a $ \lvert  y\rvert ^{\eta} $-net in $ [1,2]$, it clearly suffices to show  this modulus of continuity result.  

\begin{proposition}\label{p:full} 
Subject to $  (\tfrac1{r}, \tfrac 1{s}) $ satisfying  the hypotheses of Theorem~\ref{t:full_improve}, 
there is a $ \eta >0$ so that for all $ 0< \delta < \frac 12 $, we have 
\begin{equation}\label{e:full0}
\bigl\lVert \sup _{ \substack{ s,t\in [1,2]\\  \lvert  s-t\rvert< \delta  }}
 \lvert  A _{t} f - A_s f \rvert  \bigr\rVert_ {s} \lesssim \delta ^{\eta } \lVert f\rVert _{r}.   
\end{equation}
\end{proposition}

\subsubsection*{The Proof in Dimensions $ n\geq 3$}
It suffices to prove a version of \eqref{e:full0} at the point $ (\frac12 ,\frac12)$, and then interpolate to the other points in the interior of $ \mathbf F'_n$.  
Using \eqref{e:sig} and Plancherel,  we see that there is a full derivative in $ t$: 
\begin{align*}
\lVert  \partial_t A_t f \rVert _{L ^2 (\mathbb R ^{n} \times [1,2))} \lesssim \lVert f\rVert_2 . 
\end{align*}
It follows that for each $x \in \mathbb R ^{n}$,  $ A_t f (x) $ continuously embeds 
as a function of $ t$ into the class $ \operatorname {Lip} (\frac 14)$, so that \eqref{e:full0} follows.

\subsubsection*{The Proof in Dimension $ n=2$} 
We rely upon the detailed analysis of   Sanghyuk Lee \cite{MR1949873}, which refines the work of Schlag \cite{MR1388870} and Schlag-Sogge \cite{MR1432805} in the convolution setting. 
Again, we prove the estimate \eqref{e:full0} at a single point in  the triangle $ \mathbf F_2'$, and obtain the result as stated by interpolation.  

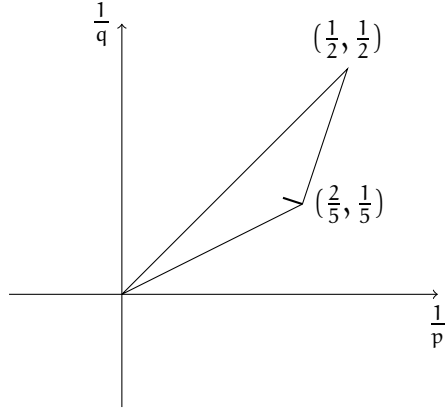
\begin{figure}
\begin{tikzpicture}[scale=6] 
\draw[->] (-.25, 0) -- (.7,0) node[below] {$ \frac 1p$}; \draw[->]  (0,-.25) -- (0,.6) node[left] {$ \frac 1q$}; 
\draw (0,0) -- (.5,.5) node[above] {$ (\frac 12 ,\frac 12)$} -- (.4,.2) node[below,right] {$ (\frac 25,\frac 15)$} -- (0,0); 
\draw[thick]  (.357,.214) -- (.4,.2); 
\end{tikzpicture}
\caption{The Triangle $ \mathbf F_2'$, and the elements of the proof of Theorem~\ref{t:full_improve} in the case of dimension $ 2$.  The thick line inside the triangle goes from $ (\frac 5{14},\frac 3 {14})$ to $ (\frac 25,\frac 15) $. 
The point $  (\frac 13, \frac 29)$ is on the thick line, and we will interpolate between an estimate at that point, and an estimate at $ (\frac 12 , \frac 12)$. }  
\label{f:F2}
\end{figure}

A Littlewood-Paley decomposition is needed.  Let $  \mathbf 1_{ [1,2]} \leq \zeta \leq \mathbf 1_{[\frac 12, \frac 52]} $ be a smooth function on $ \mathbb R $ so that $ \sum_{j\geq 1} \zeta (y/2 ^{j}) = 1$, if $  \lvert  y\rvert \geq 4 $. 
Then set $ \zeta _0 = 1 - \sum_{j\geq 1} \zeta (y/2 ^{j})$.  
For $ f \in L ^2 (\mathbb R ^{n})$, set 
Set $ \widehat {f_j}  (\xi )= \zeta (\lvert  \xi \rvert/2 ^{j} )  \widehat f (\xi )  $, for $ j\geq 1$, 
and $ \widehat {f_0} = \zeta _0  \widehat f   $.  

Let $ \mathcal M _{\delta }$ be the maximal function in \eqref{e:full0}, and let $ \mathcal M _{\delta ,j} f = \mathcal M _{ \delta } f_j$. 
We have 
\begin{equation}  \label{e:Mk}
\mathcal M _{ \delta } f \leq \sum_{j\geq 0} \mathcal M _{ \delta , j }f. 
\end{equation}
Now, it follows from \cite{MR1949873}*{just above eqn (1.5)}, 
\begin{equation}\label{e:M_Lee}
\lVert M _{\delta ,j} \;:\; L ^{p} \mapsto L ^{q}\rVert \lesssim 2 ^{j (1-\frac 5q)}, 
\qquad \tfrac 1p +\tfrac 3q =1,\  q > \tfrac {14}3. 
\end{equation}
The exponent on $ j$ above is negative for $ \frac {14}3 < q < 5$.  At  $ q=5$, we have $ (p,q)=( \frac 52,5)$, which corresponds to  the crucial vertex $(\frac 25,\frac 15)$ of the triangle $ \mathbf F'_2$.   See Figure \ref{f:F2}. 

It again follows from \eqref{e:sig} that 
\begin{align}
 \label{e:Mk2}
\lVert  \partial_t  A_t f_j  \rVert_ {L ^2 (\mathbb R ^2 \times [1,2))}  &\lesssim   2 ^{\frac j2} \lVert f\rVert_2. 
\end{align}
As a consequence, $  A_t f_j $ continuously embeds into $ \operatorname {Lip} (\frac 14)$ with norm at most $ 2 ^{j/2}$.  That is, we have the bound 
\begin{equation*}
\lVert M _{\delta ,j} \;:\; L ^{2} \mapsto L ^{2}\rVert \lesssim \delta ^{\frac 14} 2 ^{\frac j2}. 
\end{equation*}
Interpolation with \eqref{e:M_Lee}, say with  $ p=3, q = \frac 92$, shows that 
$ (\frac 1p, \frac 1q)$ sufficiently close to $ (\frac 13, \frac 29) $, we have for a positive choice of $ \eta >0$, 
\begin{equation*}
\lVert M _{\delta ,j} \;:\; L ^{p} \mapsto L ^{q}\rVert \lesssim \delta ^{ \eta } 2 ^{ - \eta j}.  
\end{equation*}
This is summable in $ j\geq 0 $, so completes our proof.

\section{Sharpness of the Sparse Bounds } 

Sharpness of the sparse bounds is not immediate from the sharpness of the $L^p$ improving estimates, as the sparse bound is defined as the largest possible sparse bound.  Nevertheless, sharpness will follow from the examples that show that the $ L ^{p}$ improving estimates are sharp.  

\begin{proposition} \label{p:sharp} Suppose that for $1\leq r,s< \infty$ satisfy $\frac1r+\frac 1s\geq 1$. 
\begin{enumerate}
\item If the sparse bound $ \lVert M _{\textup{lac}}  : (r, s)_m \rVert < \infty$  holds, 
then,  $ (\frac 1r, \frac 1s) \in \mathbf L_n$, where the last set is the triangle defined in Theorem~\ref{t:lac}. 

\item If the sparse bound $ \lVert M _{\textup{full}}  : (r, s)_m \rVert < \infty$  holds, then, $(\frac 1r, \frac 1s) \in \mathbf{F}_n$, where the latter set is the trapezium defined in Theorem \ref{t:full}.  
\end{enumerate}
\end{proposition}

We recall this  elementary fact, \cite{2016arXiv161001531L}*{Lemma 4.7}. 
For all $ 1\leq r, s < \infty $,  there is a constant $ C$ so that for all   $ f$ and $ g$, 
there  a sparse form $ \Lambda _{0}$ so that 
\begin{equation} \label{e:ONE}
\sup _{\mathcal S} \Lambda _{\mathcal S, r, s} (f,g) \leq C \Lambda _0 (f,g).  
\end{equation}
For the pairs $ f, g$ that we describe below, it will be very easy to verify this principle.  The largest  sparse form $ \Lambda _0$ will consist of a single cube, 
namely  one that contains the support of the functions defined below, and is of minimal side length. 

\smallskip 

\begin{figure}
\begin{tikzpicture}
\filldraw (0,0) circle (.1em) node[right] {$ x$}; 
\draw[dashed] (.05,.05) circle (3em); 
\draw   (0,0) circle (3em) ; \draw   (0,0) circle (3.25em) ; 
\draw (1.15,1.15) node {$ f_ \delta $}; 
\end{tikzpicture}

\caption{The example showing sharpness of the bounds in Theorem~\ref{t:littman}. The function $ f_ \delta $ is the indicator of the thin annulus, of width $ \delta $.  For a point $ x$ within say $ \delta /2$ of the center of the annulus, one has $ A_1 f_ \delta  (x) \geq c$.  The dashed circle is centered at $ x$, and has radius $ 1$. At least $ \frac 14 $ of the dashed circle is inside the support of $ f _{\delta }$. This leads the inequality \eqref{e:zlac}.}  
\label{f:littman}
\end{figure}
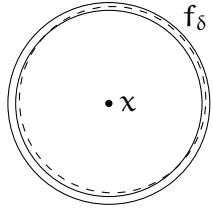

 \begin{proof}[Proof of Proposition \ref{p:sharp}.1.]
 We begin with the lacunary maximal operator, $  M _{\textup{lac}}$,  and the $ L ^{p}$-improving bounds of Littman \cite{MR0358443} and Strichartz 
\cite{MR0256219}.  For $ 0 < \delta < \tfrac 14$, let $ f_ \delta= \mathbf 1_{ \lvert \,\lvert  x\rvert-1 \rvert < \delta }$ be the indicator of a thin annulus around the unit circle.  
Note that for small absolute constant $ c$, that we have 
\begin{equation}\label{e:LS1}
 A_1 f_\delta  (x)\ge  c g_\delta  (x) =c\mathbf 1_{  \lvert  x\rvert < c \delta  }.  
\end{equation}
This  example is illustrated in Figure \ref{f:littman}. It   establishes the sharpness of exponents $ r$ and $ s$ in  Theorem~\ref{t:littman}.  
Suppose that $ M _{\textup{lac}}$ satisfies an $ (r,s)$-bound, where $ 1/r+1/s >1$.   We then have 
\begin{equation*}
\delta ^{n} \lesssim \langle A_1 f_\delta,  g_\delta   \rangle 
 = \langle  f_\delta , A_1 g_\delta   \rangle
\lesssim \min\{\Lambda _{\mathcal S, r,s} (f_\delta ,g_\delta ) , \Lambda _{\mathcal S', s,r} (f_\delta ,g_\delta )) .  
\end{equation*}
for some choice of sparse collections $ \mathcal S$ and $ \mathcal S'$. 
Note that we have two bounds on the right, due to the convolution structure of the question. 

But each cube in the collections $ \mathcal S$ and $ \mathcal S'$ should intersect the support of $ f$ and of $ g$.  That is, we can assume that $ \{ x \;:\; \lvert  x\rvert <2 \} \subset  Q$, for each $Q\in \mathcal{S}$.  But then, the contribution of such cubes decreases as the side length of the cube increases. 
So, it suffices  to have $ \mathcal S$ to consist of just  a single cube $ Q$ of side length, 2 say. 
Our assumption leads to the conclusion 
\begin{equation*}
\delta  ^{n} \lesssim 
 \langle A_1 f_\delta,  g_\delta   \rangle 
\lesssim  \min \{\lVert f_\delta \rVert_r \lVert g_\delta \rVert_s, 
\lVert f_\delta \rVert_s \lVert g_\delta \rVert_r\}
\lesssim  \delta ^{ \max \{ \frac 1r+ \frac ns,  \frac nr+\frac 1s\} }.    
\end{equation*}
We conclude that we need to have the inequality below, which  tells us that $ (\frac 1r, \tfrac 1s) \in \mathbf L_n$.   
\begin{equation} \label{e:zlac}
 \max \{ \tfrac 1r+ \tfrac ns,  \tfrac nr+\tfrac 1s\} \leq n . 
\end{equation}
 And so, we cannot do better than the $ L ^{p}$-improving bounds Littman and Strichartz for the lacunary maximal function.  
 \end{proof}

\medskip  
\begin{figure}
\begin{tikzpicture}
\draw  (0,0) rectangle (0.5,4) node[right] {$ R_1$} ;  
\draw[|<->|] (0,-0.25) -- (0.5, -0.25) node[below,midway] {$C \delta $}; 
\draw[|<->|]  (-.3, 0) -- (-.3,4)  node[midway,left] {$ C\sqrt \delta $}; 
\draw (5,1) rectangle (9,3) ;  \draw (7, 3.25) node {$ R_2$}; 
\draw [|<->|] (5,.75) -- (9,.75) node[midway,below] {$ \approx 1$}; 
\draw [|<->|] (.55,2) -- (4.95,2) node[midway,below] {$ \approx 1$}; 
\draw [|<->|] (9.25,1) -- (9.25,3) node [midway,right] {$  \sqrt \delta $}; 
\filldraw (6,2) circle (.1em) node[right]  {$ x$}; 
\draw[dashed]  (.5,-.5) arc (190:170:15);  
\end{tikzpicture}
\caption{An example for the operator $ \tilde M$.  The rectangle $ R_1$ is on the left, and 
at each point $ x \in R_2$, there is a circle of radius $ 1\leq r \leq 2$ which intersects a substantial portion of the rectangle $ R_1$, as indicated by the dashed arc of a circle.
We have $ \tilde M \mathbf 1_{R_1} (x) \gtrsim \delta ^{\frac {n-1}2 }$. The assumed $ (r,s)$ bound leads to the restriction \eqref{e:quad}.} 
\label{f:full}
\end{figure}
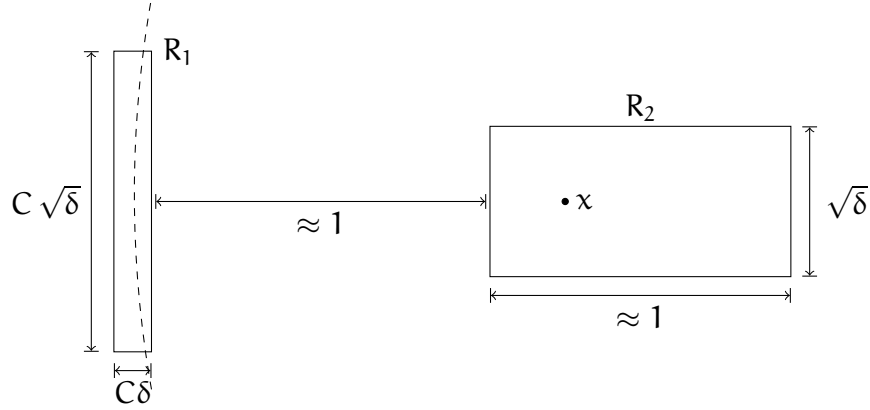

\begin{proof}[Proof of Proposition \ref{p:sharp}.2.]
We turn to the case of the full spherical maximal function.  The sharpness of the trapezium in Theorem~\ref{t:full_improve} is given by three examples.  One of these is the thin annulus example just used, and this demonstrates the sharpness along the line from $ P_1 = (0,1)$ to $ P_4 = (\frac {n ^2 -n} {n ^2 +1 }, \frac {n^2-n+2} {n ^2 +1})$.  Here, we are referring to the trapezium $\mathbf{F}_n$ in Figure \ref{f:F}.  

The second example is a Knapp type example illustrated in Figure~\ref{f:full}.  Define two rectangles by 
\begin{equation*}
R_1 = [-C\sqrt \delta , C\sqrt \delta ] ^{n-1} \times [-C \delta , C\delta] , \qquad 
R_2 = [-\sqrt \delta , \sqrt \delta ] ^{n-1} \times [ \tfrac 43, \tfrac 53 ].  
\end{equation*}
Then, note that the localized maximal function $ \tilde M $ applied to  $ \mathbf 1_{R_1}$  satisfies 
$
\tilde M \mathbf 1_{R_1} \gtrsim  \delta ^{ \frac {n-1}2} \mathbf 1_{R_2} 
$.
Then, assuming the $ (r,s)$ sparse bound for the full maximal function, we have 
\begin{equation*}
\delta ^{ n-1} \lesssim \langle \tilde M f, g\rangle 
\lesssim \Lambda _{\mathcal S , r,s} (\mathbf 1_{R_1}, \mathbf 1_{R_2}).  
\end{equation*}
The sparse form on the right is largest, up to a constant, taking $ \mathcal S$ to consist of  a single cube of bounded side length, which contains  the two rectangles $ R_1$ and $ R_2$.  We deduce that 
\begin{equation*}
\delta ^{ n-1} \lesssim  \lvert  R_1\rvert ^{1/r} \lvert  R_2\rvert ^{1/s} \lesssim \delta ^{\frac {n+1} {2r}+ \frac {n-1} {2s}}.  
\end{equation*}
From this, we see that we necessarily must have 
\begin{equation} \label{e:quad}
\tfrac {n+1} {r}+ \tfrac {n-1} {s} \leq 2 (n-1). 
\end{equation}
This gives the restriction on the line from the point $P_4 $ to $ P_3 = (\frac {n-1}{n}, \frac {n-1}{n})$.  

\smallskip 

A third example of Stein is the function  $ h (x) = \mathbf 1_{ \lvert  x\rvert<1 } \lvert x  \rvert ^{1-n}  (\log \lvert  x\rvert ) ^{-1}  $, we have $ M  _{\textup{full}}h (x)$ is infinite on a set of positive measure. Hence, $ M_{\textup{full}} $ is unbounded on $ L ^{p}$, for $ 1< p \leq \frac n {n-1}$. Now, if $ M _{\textup{full}}$ satisfies an $ (r,s)$ bound for any $ 1< r \leq \frac n {n-1}$ and any finite $ s$, it would follow that $ M _{\textup{full}}$ is of weak-type $ L ^{r}$, which is impossible.   This shows the sharpness of the line from $ P_2$ to $ P_3$.  
\end{proof}

These examples also show that the `continuity' condition can not hold at the critical indexes for the $L^p$ improving inequalities.  

\begin{proposition}   Suppose that for $1\leq r,s< \infty$ satisfy $\frac1r+\frac 1s >  1$. 
\begin{enumerate}
\item If the  inequality \eqref{e:littman2}  holds,  then $ (\tfrac 1r, \tfrac 1s)$ is in the interior of $ \mathbf L_n$, 
the triangle defined in Theorem~\ref{t:lac}. 

\item If the inequality \eqref{e:full_improve} holds,  then, $(\frac 1r, \frac 1s) $ is in the interior of $ \mathbf{F}_n$, where the latter set is the trapezium defined in Theorem \ref{t:full}.  
\end{enumerate}
\end{proposition}

\begin{proof}
This is a corollary to the fact that the relevant examples in the $L^p$ improving estimates are supported on small sets. 

1. Suppose  that $ (\tfrac 1r, \tfrac 1s)$ is on the boundary of $ \mathbf L_n$, which is to say that it satisfies equality in \eqref{e:zlac}.    We have the assumed inequality  \eqref{e:littman} with $\lvert y\rvert $ much smaller than one. Apply it to the function $f_\delta$, where $\delta$ is much smaller than $\lvert y\rvert $.  It follows that there is no cancellation after translation by $ y$, so that 
\begin{align*}
\lVert A_1  f_ \delta  - \tau _y A_1 f_\delta \rVert_s 
\simeq \lVert f_\delta\rVert_r \lesssim \lvert  y\rvert ^{\eta }  \lVert f_\delta\rVert_r .  
\end{align*}
This is a contradiction. 

\smallskip 

2.  Suppose  that $ (\tfrac 1r, \tfrac 1s)$ is on the boundary of $ \mathbf F_n$, and that we have the assumed inequality \eqref{e:full_improve}.  It follows from the first part of the argument that $  (\tfrac 1r, \tfrac 1s)$ 
cannot lie on the line from $P_1$ to $ P_4$, where we are referring to the points in Figure~\ref{f:F}.  
By the example of Stein described above, it cannot lie on the line from $ P_2$ to $ P_3$.  And, by a similar argument to the one above, but using the example from Figure~\ref{f:full}, it also follows that $  (\tfrac 1r, \tfrac 1s)$ cannot lie on the line from $ P_3$ to $ P_4$.  This is a contradiction, so the argument is complete.  

\end{proof}
\section{Weighted Inequalities} \label{s:weighted}

The maximal function $ M _{\textup{lac}}$ applied to the indicator of a ball  $ B$ of radius $ 1$ centered at the origin 
is dominated by 
\begin{equation*}
M _{\textup{lac}} \mathbf 1_{B} (x)\lesssim 
\mathbf 1_{2B} (x) + \sum_{k=1} ^{\infty }  2 ^{- k (n-1)} \mathbf 1_{  \lvert  \lvert  x\rvert-2 ^{k} \rvert\leq 2 }.
\end{equation*}
Thus, there is no reason to think that Muckenhoupt weights are the correct tool to understand the behavior of this (or the full) spherical maximal function in weighted spaces.
(See Figure~\ref{f:full} for an example showing that the full supremum is poorly adapted to Muckenhoupt weights.) 

Nevertheless, the question of weighted inequalities for weights of Muckenhoupt type has attracted interest \cites{MR1373065,MR1922609}. And the sparse bounds are especially efficient  for such weights. We detail here some of the implications of our main theorems in this direction. We will see  that our sparse bound contains the best known prior bound for $ M _{\textup{full}}$, and yields new information.  The full implications would be a little technical, and so we do not develop them here. 

We indicate here how easy it is to prove $L^p$ bounds for sparse forms, and leave the details of the weighted case to the references. 
The familiar $L^p$ bounds for the spherical maximal functions are seen to trivially follow from our sparse bounds.  

\begin{proposition}  \label{p:pp}
Let $1\leq r < p < s'<\infty$.  We have the inequality 
\begin{equation}
\Lambda_{r,s}(f,g) \lesssim \lVert f\rVert_p \lVert g\rVert_{p'} . 
\end{equation}
\end{proposition}

\begin{proof}
The notation for the sparse form is in \eqref{e:Lam}. Recall that to each cube $Q$ in the sparse collection $\mathcal{S}$, there is a set 
$E_Q\subset Q$, with $\lvert E_Q\rvert \geq \frac 12 \lvert Q\rvert$, so that the sets $\{E_Q \,:\, Q\in\mathcal{S}\}$ are pairwise disjoint. 
Thus 
\begin{align*}
\Lambda _{r,s} (f,g) &< 2 \int \sum_{Q\in \mathcal S} \mathbf 1_{E_Q}  \langle f \rangle _{Q,r} \langle g \rangle _{Q,s}   \; dx 
\\
&\leq \int  M_r f \cdot M_s g \;dx \lesssim \lVert M_r f\rVert_p \lVert M_s g\rVert _{p'} \lesssim \lVert f\rVert_p \lVert g\rVert _{p'}. 
\end{align*}
Above $M_r f = \sup_Q \langle f \rangle _{Q,r}  \mathbf 1_{Q}$ is the maximal function with $r$th powers.  
\end{proof}

A \emph{weight} is a function $ w (x) >0$ a.e., which is the density of a measure on $ \mathbb R ^{n}$, also written as $ w (E ) = \int _{E} w \;dx$.  For $ 1< p < \infty $, the dual space to $ L ^{p} (w)$ (with respect to Lebesgue measure) is $ L ^{p'} (\sigma )$, where $ p' = \frac p {p-1}$ and $ \sigma = w ^{1-p'}$.   Note that $ w \cdot \sigma ^{p-1} \equiv 1$.  
A weight $ w \in A_p$ if this equality holds in an average sense, uniformly over all locations and scales. Namely, define 
\begin{equation*}
[w] _{A_p} =\sup _{Q}  \langle w \rangle_Q \langle \sigma  \rangle_Q ^{p-1}  < \infty , \qquad  \sigma = w ^{1-p'}.  
\end{equation*}
Above, the supremum is over all cubes $ Q$.    At $ p=1$, we define 
\begin{equation}\label{e:A1}
[w] _{A_1} =\sup _{Q} \sup _{x\in Q} \frac {  \langle w \rangle_Q } {w (x)}. 
\end{equation}
A weight $ w $ is in the \emph{reverse H\"older} class $ RH _{r}$,  $1\leq r < \infty$,  if 
\begin{equation*}
[w] _{RH _{r} }  =\sup _{Q}  \frac {\langle w \rangle _{Q,r}} {\langle w \rangle_Q} < \infty .  
\end{equation*}
Qualitatively, the conditions of a weight $w$ being in the intersection of $A_p$ and reverse H\"older spaces is the same as $w$ having a factorization $w \in A_1 ^{\alpha} A_1^{\beta} = 
\{ u_1 ^\alpha u_2 ^\beta \;:\; u_1, u_2 \in A_1\}$. 
This is made precise in this proposition.

\begin{proposition}\label{p:factor} Let $ u_1 , u_2 \in A_1$, and let $ \rho >0$, and $ 1< r < p < \infty $.  
We have 
\begin{equation}\label{e:factor}
A_1 ^{\frac 1 \rho } A_1 ^{-\frac pr +1  } = A _{\frac pr } \cap RH _{\rho }.  
\end{equation}
\end{proposition}

\begin{proof}
These two facts are well known.
(1) A weight in $ A_p$ can be factored into the product of $ A_1$ weights 
\begin{equation}\label{e:1factor}
w \in A_p  \quad \Longleftrightarrow \quad  w = u_1u_2 ^{1-p}, \quad u_1 u_2 \in A_1. 
\end{equation}
(2) The condition $ w\in  A _{p/r } \cap RH _{\rho }$ is equivalent to $ w ^{\rho } \in A _{\rho (p/r-1)+1}$.  
Combining these two facts proves the proposition. 
\end{proof}

We focus on qualitative aspects of weighted inequalities for the sparse maximal functions.  
While quantitative estimates are available, and not too hard to prove, we think that what we can prove right now is improvable. 
(See \S \ref{s:msparse}.)  Set $\mathcal {L}_p$ to be those weights $w$ for which $M _{\textup{lac}}$ maps $L^p(w)$ to $L^p(w)$, 
for $1< p < \infty$.  Use the same type of notation $\mathcal{F}_p$ for $M _{\textup{full}}$.

 We have these two corollaries 
to our sparse bounds for the lacunary and full spherical maximal operators. These are obtained by combining our main theorems with the bounds in Theorem \ref{t:bfp}.  
As we only seek qualitative results, and the conditions of $ A_p$ and $ RH _{r}$ are open, 
we are free to work on the boundary of the figures $ \mathbf L_n$ and $ \mathbf F_n$.  
 See Figure~\ref{f:corollary} for graphs of the two functions introduced below. 

\begin{corollary}\label{c:wtd}  For the lacunary and full spherical maximal function, we have these two sets of weighted inequalities. 

\begin{enumerate}
\item  Define $ \frac 1 {\phi _{\textup{lac}} (1/r)}$ to be a piecewise linear function on $ [0,1]$ whose graph connects the points  $ Q_1 = (0,1)$, $ Q_2 = (\frac  {n-1}n, \frac {n-1}n)$, and $ Q_3 = (1,0)$.  That is, 
\begin{equation*}
\frac 1 {\phi _{\textup{lac}} (1/r)} = 
\begin{cases}
1- \frac 1 {rn}   &  0< \frac 1r \leq \frac n{n+1}, 
\\
n (1- \tfrac 1{r})   &   \frac n{n+1} < \frac 1r < 1 . 
\end{cases}
\end{equation*}
Assuming $ 1 < r < p < \phi (r)'$, we have 
\begin{gather} \label{e:LL}
{A _{p/r}} \cap {RH _{ (\phi  _{\textup{lac}}(r)'/p)'}}   
\subset \mathcal{L}_p . 
\end{gather}

\item   Define $ \frac 1 {\phi _{\textup{full}} (1/r)}$ to be the piecewise linear function on $ [0, \frac {n-1}n]$ whose graph connects the points  $ P_1 =(0,1)$, 
$ P_4 = (\frac {n ^2 -n} {n ^2 +1}, \frac {n ^2 -n+2} {n ^2 +1})$
and $P_3 = (\frac {n-1} {n},  \frac {n-1}n)$.
Assuming $ \frac {n} {n-1} < r < p < \phi _{\textup{full}} (r)'$, we have 
\begin{equation}  \label{e:FULL}
{A _{p/r}} \cap {RH _{ (\phi  _{\textup{full}}(r)'/p)'}}   
\subset \mathcal{F}_p  . 
\end{equation}

\end{enumerate}
\end{corollary}

\begin{figure}
\begin{tikzpicture}

\begin{scope}[scale=1.5]
\draw[->,thick] (-.5,0) -- (2.5,0) node[below] {$ \frac1r$} ;
\draw[->,thick] (0,-.5) -- (0,2.75) node[left] {$ \frac1s$}; 
\draw[thick] (0,2)   --  (1.7,1.7) -- (2,0)  node[midway, right] {$ 1/\phi _{\textup{lac}}$}  ; 
\filldraw  (1.7,1.7) circle (.1em); 
\draw (2,1.8) node[above] { $  (\frac n {n+1},\frac n {n+1})$}; 
\draw (2,0) node[below] {1}; \draw (0,2) node[left] {1}; 

\begin{scope}[xshift=4.5cm] 
\draw[->,thick] (-.5,0) -- (2.5,0) node[below] {$ \frac1r$} ;
\draw[->,thick] (0,-.5) -- (0,2.5) node[left] {$ \frac1s$}; 
\draw[thick] (0,2)  --  (1.4,1.8) -- (1.9, 1.5 );  \filldraw  (1.4,1.8) circle (.1em); \filldraw  (1.9, 1.5) circle (.1em); 
\draw (0,2)  node[left] {$ P_1$}; 
\draw (2.4 ,1.8)  node[above] {$  { P_4= (\frac {n ^2 -n} {n ^2 +1}, \frac {n ^2 -n+2} {n ^2 +1})} $}; 
\draw  (1.9, 1.5 )  node[right] { $ P_3= (\frac {n-1} {n},  \frac {n-1}n)$};  \draw (1,1) node {$ 1/\phi _{\textup{full}}$};
\draw[dashed] (0,2) -- (1.9, 1.5 );
\end{scope}

\end{scope}

\end{tikzpicture}

\caption{The two functions $  1/\phi _{\textup{lac}}$ and $ 1/\phi _{\textup{full}} $ of Corollary~\ref{c:wtd}. 
The dashed line is the function $ 1/ \psi $, the function in \eqref{e:zc}.}
\label{f:corollary}
\end{figure}
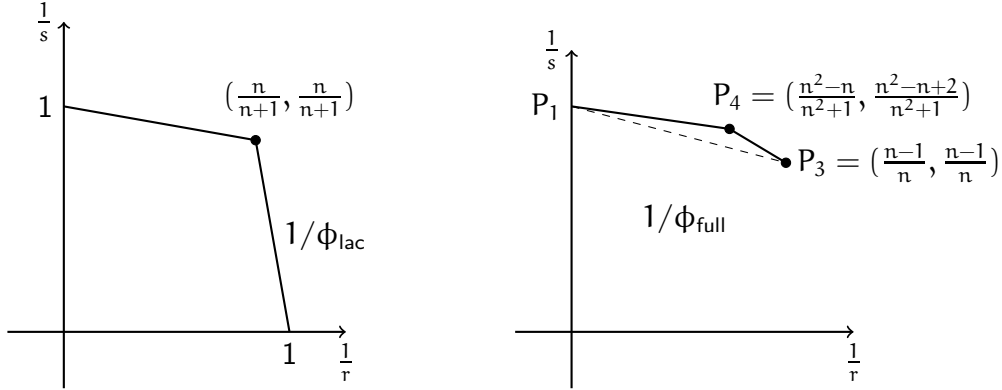

The case of radial weights has been completely analyzed by Duoandikoetxea and Vega \cite{MR1373065}.  Here, we recall this result, which records the possible inequalities for radial weights. These are sharp, except possibly the $ a = 1-n$ endpoint case in \eqref{e:DV2}.  (In particular, this shows that the class $\mathcal{L}_p$ 
does not satisfy the classical duality $\mathcal{L}_{p'}= \mathcal{L}_p^{1-p'}$. See \cite{MR1373065} for more details.) 
\begin{priorResults}\label{t:}\cite{MR1373065} 
Let $ w _{a} (x)= \lvert  x\rvert ^{a} $ be a radial weight on $ \mathbb R ^{n}$, for $ a \in \mathbb R $. 
We have the inequalities below, for $ 1< p < \infty $.  
\begin{gather} \label{e:DV1}
 w_a \in \mathcal{L}_p,  \qquad  1-n \leq a < (n-1) (p-1),  
\\ \label{e:DV2}
 w_a \in \mathcal{F}_p,   \qquad   1-n < a < (n-1) (p-1) -n. 
\end{gather}
In \eqref{e:DV2}, the restriction on $ a$ implies that $  \tfrac {n} {n-1} < p < \infty $. 
\end{priorResults}

We cannot recover the full strength of this theorem.  But this is to be expected: the category of $A_p$ weights is not the correct one to characterize the weights for the spherical maximal function, and our sparse results are sharp.  This suggests that the sparse bounds are proving the sharpest possible results in the category of Muckenhoupt type weights.  
We can 
improve upon  result below of Cowling, Garcia-Cuerva and Gunawan \cite{MR1922609}. 
It gives sufficient conditions for $ M _{\textup{full}}$ to satisfy a weighted inequality in terms of a factorization of 
the weight.  

\begin{priorResults}\label{t:CD} \cite{MR1922609}*{Thm 3.1}  
Let $ \frac n {n-1} < p < \infty $, and $ \max \{0, 1-\frac pn\} \leq \delta < \frac {n-2} {n-1}$. 
Then  $A_1^\delta A_1 ^{\delta (d-1)- (d-2)} \subset \mathcal{F}_p$. 
\end{priorResults}

We will deduce this as a special case of \eqref{e:FULL}. 

\begin{proof}[Proof of Theorem~\ref{t:CD}]  Rather than use the exact form of $ \phi _{\textup{full}}$ in \eqref{e:FULL}, 
we use the restricted form 
\begin{equation} \label{e:zc}
\psi (r) ^{-1} = 1 - \tfrac 1 {r (n-1)}  , \qquad  \tfrac n {n-1} < r < \infty .  
\end{equation}
It follows that we have a sparse form bound $ (r, \psi (r))$. 
This function corresponds to the dashed line in Figure \ref{f:corollary}.
Provided $ r < p < \psi (r)' = (r (n-1))' =: s'$, we have a weighted inequality, for $ w \in A _{p/r} \cap RH  _{ (s'/p)'}$.  
Now, $ (s'/p)' = \frac {r (n-1)} {r (n-1) -p} = 1 - \frac p {r (n-1)}$. By Proposition~\ref{p:factor}, we have 
$
A_1 ^{ 1 - \frac p {r (n-1)}} A_1 ^{1- \frac pr} \subset \mathcal{F}_p
$.
Setting $ \delta = 1 - \frac p {r (n-1)}$, we have $ 1-\frac pr = \delta (n-1) - (n-2)$.  This matches the conclusion of the Theorem, so the proof is complete. 
\end{proof}

As the proof above indicates, stronger results than those of Theorem \ref{t:CD} hold.  The authors of \cite{MR1922609} raised the possibility 
that $  A _1^{1 - \frac 1n} \subset \mathcal{F}_p$. Here, we show that this is indeed the case, provided $ p$ is sufficiently large.  It will be clear that more is true, but we do not pursue the details here. 

\begin{proposition}\label{p:n} For $ n\geq 2$,  we have 
$A_1^{1-\frac1n} \subset \mathcal{F}_p$, for 
$\tfrac {n ^2 +1}  {n ^2 -n }< p  < \infty $.
\end{proposition}

\begin{proof}
We use the proof strategy for Theorem~\ref{t:CD}, but use the sparse bound provided to us by 
the point $ P_4  = (\frac {n ^2 -n} {n ^2 +1}, \frac {n ^2 -n+2} {n ^2 +1})$. 

Indeed, assuming a sparse bound of the form $ (r_0, s_0)$, we have the inequality 
\begin{equation} \label{e:ZR}
\lVert M _{\textup{full}} :  L ^{p} (w) \mapsto L ^{p} (w)\rVert < \infty , \qquad   w = u ^{1/ \rho }, \ u\in A_1, 
\end{equation}
provided  $ r_0 < p < s_0'$, and  $ \rho = (s_0'/p)'$. 

Setting $ (1/r_0,1/s_0) = P_4$, we have 
\begin{align*}
\frac 1 {s_0} = \frac {n ^2 -n+2} {n ^2 +1},  && \frac 1 {s_0'} = \frac {n-1} {n ^2 +1}, 
\\
\frac 1 {r_0}= \frac {n ^2 -n} {n ^2 +1} ,  && \frac {s_0'} {r_0} = n. 
\end{align*}
It follows that $ \rho =  (s_0'/p)' = \frac n {n-1}$. For $ p> r_0$, we are allowed to take $ w = u ^{\frac 1 {\rho }}= u^{1 - \frac 1n}$, as claimed, provided $ p> r_0$.  
\end{proof}

\section{Further Remarks} 

\subsection{Endpoint Issues}
Richard Oberlin  \cite{170404297} has investigated the endpoint issues. Namely, for a class of Radon transforms, a sparse bound is proved at the boundary of the sparse region.  The `local $L^r$ norm' is 
adjusted with a logarithmic factor.  It would be interesting to further develop the endpoint estimates.

\subsection{Weighted Estimates for m-sparse forms}\label{s:msparse}

 For $ 1< p < \infty $, the dual space to $ L ^{p} (w)$ (with respect to Lebesgue measure) is $ L ^{p'} (\sigma )$, where $ p' = \frac p {p-1}$ and $ \sigma = w ^{1-p'}$.  This is referenced in the statement of the Theorem below, which gives weighted inequalities for sparse forms. 
 These estimates are sharp in the Muckenhoupt and reverse H\"older indices.  

\begin{priorResults}\label{t:bfp}\cite{MR3531367}*{\S6} 
Let $ 1\leq r < s' < \infty $. Then, 
\begin{gather} \label{e:BBP}
\Lambda _{r,s} (f,g)  
\leq \bigl\{ [w] _{A _{p/r}} \cdot  [w] _{RH _{ (s'/p)'}}\bigr\} ^{\alpha }  \lVert f\rVert _{L ^{p} (w)} \lVert g \rVert _{L ^{p'} (\sigma )},  \qquad  r < p < s', 
\\ \label{e:alpha}
\textup{where} \quad 
\alpha = \max \Bigl\{ \frac 1 {p-1}, \frac {s'-1} {s' -p} \Bigr\} . 
\end{gather}

\end{priorResults}

For sparse forms of type $ (1,1)$, we recall that we have these estimates. 
\begin{align*}
\Lambda _{1,1} (f, g) &\lesssim [w] _{A_p} ^{\max \{1, \frac 1 {p-1}\}} \lVert f\rVert _{L ^{p} (w)} \lVert g\rVert _{L ^{p'} (\sigma )} , 
\\
\Lambda _{1,1,m} (f, g) &\lesssim [w] _{A_p} ^{\frac 1 {p-1}} \lVert f\rVert _{L ^{p} (w)} \lVert g\rVert _{L ^{p'} (\sigma )} . 
\end{align*}
Both estimates are well-known.  A very nice proof of the first bound can be found in \cite{MR3000426}.  The second follows from a comparison to the maximal function, namely Buckley's inequality \cite{MR1124164}.  Thus, the sparse forms and the $ m$-sparse forms can obey different weighted estimates.  

The papers \cites{MR3531367,MR3591468} supply explicit and  sharp estimates for $ (r,s)$-sparse forms.  But, they do so only for the form \eqref{e:sparse_def}, with $ F_Q \equiv Q$.  As this paper indicates, obtaining the sharp estimates for the $ m$-sparse forms is also interesting. 

\subsection{Sharpness of the Weighted Estimates}   

We conjecture that the bounds in Corollary \ref{c:wtd} are sharp in the category of weights allowed. 
For the sake of clarity, let us state a conjecture for the lacunary maximal function.  

\begin{conjecture} Using the notation of  Corollary \ref{c:wtd}, this holds.  
Let $1< r < p < \phi_{\textup{lac}}(r)'$, and set $\rho = (\phi_{\textup{lac}}(r)'/p)'$.  
If $ 1/\rho < \alpha$, then there is a weight $w= u_1 ^{\alpha} u_2 ^{-\frac pr+1}$, for weights $u_1,u_2\in A_1$, so that $M _{\textup{lac}}$ is not bounded on $L^p(w)$.  
\end{conjecture}

\subsection{The Endpoint Estimate}

A result of Seeger, Tao and Wright  addresses an endpoint estimate for the lacunary spherical maximal function, showing this.  

\begin{priorResults}\label{t:lac1}  \cite{MR2058385} The lacunary maximal function $ M _{\textup{lac}}$ is bounded as a map from $ L \log\log L$ into weak $ L ^{1}$. 
\end{priorResults}

Also see the recent significant improvement by Cladek and Krause \cite{170301508}.  
The proof is based upon $ T T ^{\ast}  $ methods, and so it is tempting to think that a reading of the paper might prove sparse bound for $ M _{\textup{lac}}$ of the form $ (r,2)$, for all $ 1< r < 2$.
But such a sparse bound cannot hold. 
It is however interesting to speculate about what sparse bound the argument of \cite{MR2058385} would imply.  

\subsection{Other Themes}

\textbf{1.} As was pointed out by Duoandikoetxea and Vega \cite{MR1373065}, it is interesting to establish inequalities of  
Fefferman-Stein type, namely 
\begin{equation*}
\lVert M _{\textup{lac}} \,:\, L ^{p} (w) \mapsto L ^{p} (N w)\rVert, 
\end{equation*}
for some auxiliary maximal operator $ N$.  This has been addressed in \cite{MR3418202}. It would be interesting to extend the results of this paper.  

\smallskip 

\textbf{2.} The paper \cite{MR1922609} studies weighted inequalities from $ L ^{p}  $ to $ L ^{q}$ spaces for the maximal operator 
\begin{equation*}
\sup _{t>0} t ^{\alpha } A_t f , \qquad \alpha = n (\tfrac 1p - \tfrac 1q). 
\end{equation*}
Sparse bounds should be possible for such an operator.  

\smallskip 

\textbf{3.} Variants of the maximal operator,  formed over restricted ranges of radii of spheres have been considered. Namely, 
\begin{equation*}
\sup _{t \in E}  A_t f , \qquad  E\subset (0, \infty ). 
\end{equation*}
See \cite{MR1955209}. Subject to a dimensionality condition on $ E$, a range of $ L ^{p}$ inequalities can be proved.  Again, sparse bounds should be available in this setting.

\textbf{4.} 
The paper of Jones, Seeger and Wright \cite{MR2434308}*{Thm 1.4} prove variational results  for the full spherical maximal function.  It would be interesting to extend this bound to a sparse bound.  
Also see \cite{160405506} for the some sparse variational results.

\textbf{5.} 
Sparse bounds should hold for other Radon transforms.  Key components would be (a) an appropriate dilation structure, and (b) variants of the continuity results Theorem~\ref{t:littman} and Theorem~\ref{t:full_improve}.   Note that these will become more involved in the cases in the variable curve case, as in \cite{MR1432805}.   

\textbf{6.} Cladek and Y.\ Ou \cite{170407810} have studied sparse bounds for Hilbert transforms and averages along a general class of curves.  

\bibliographystyle{alpha,amsplain}	

\end{document}